\documentclass[12pt]{article}
\usepackage{tipa}
\usepackage[all,import]{xy}
\usepackage{graphicx,color}
\usepackage{amsmath, amsthm}
\usepackage{amsbsy}
\usepackage{amssymb}
\usepackage{cases}
\usepackage{enumerate}
\usepackage{bm}
\usepackage[colorlinks,linkcolor=black,anchorcolor= black,citecolor= black,urlcolor=black]{hyperref}
\usepackage{multirow}
\usepackage[all,import]{xy}
\usepackage{times}
\usepackage{mathptmx}

\usepackage{natbib} 
\bibpunct{(}{)}{;}{a}{}{,} 

\newtheorem{corollary}{Corollary}

\newtheorem{theorem}{Theorem}
\newtheorem{lemma}{Lemma}

\def\pr{\text{P}}

\DeclareMathOperator\logit{\textnormal{logit}}
\DeclareMathOperator\Vol{\textnormal{Vol}}
    
\makeatletter
\newcommand*{\rom}[1]{\expandafter\@slowromancap\romannumeral #1@}
\makeatother

\renewcommand{\d}[1]{\ensuremath{\operatorname{d}\!{#1}}}

\begin{document}

\setlength{\baselineskip}{2\baselineskip}

\title{\bf The Differential Geometry of Homogeneity Spaces Across Effect Scales}
\author{Peng Ding~\footnote{Department of Statistics, University of California at Berkeley, CA, USA. 
Email: \url{pengdingpku@gmail.com}} 
~and Tyler J. VanderWeele~\footnote{Department of Epidemiology and Biostatistics, Harvard T. H. Chan School of Public Health, Boston, MA, USA. 
Email: \url{tvanderw@hsph.harvard.edu}}}
\date{}

\maketitle

\begin{abstract}
If an effect measure is more homogeneous than others, then its value is more likely to be stable across different subgroups or subpopulations. Therefore, it is of great importance to find a more homogeneous effect measure that allows for transportability of research results. For a binary outcome, applied researchers often claim that the risk difference is more heterogeneous than the risk ratio or odds ratio, because they find, based on evidence from surveys of meta-analyses, that the null hypotheses of homogeneity are rejected more often for the risk difference than for the risk ratio and odds ratio. However, the evidence for these claims are far from satisfactory, because of different statistical powers of the homogeneity tests under different effect scales. For binary treatment, covariate and outcome, we theoretically quantify the homogeneity of different effect scales. Because when homogeneity holds the four outcome probabilities lie in a three dimensional sub-space of the four dimensional space, we can use results from differential geometry to compute the volumes of these three dimensional spaces to compare the relative homogeneity of the risk difference, risk ratio, and odds ratio. We demonstrate that the homogeneity space for the risk difference has the smallest volume, and the homogeneity space for the odds ratio has the largest volume, providing some further evidence for the previous claim that the risk difference is more heterogeneous than the risk ratio and odds ratio.

\end{abstract}

\section{Introduction}

Practitioners are often interested in generalizing their finding across different subgroups or populations. This transportability problem \citep{pearl2014external, keiding2016perils} relies on homogeneity of the effect measure across these subgroups or subpopulations. However, homogeneity, or equivalently the absence of interaction, is scale dependent as documented in the classical psychology literature \citep{krantz1971conjoint, bogartz1976meaning, loftus1978interpretation} and statistics literature \citep{box1964analysis, cox1984interaction, darroch1994synergism, de2007interpretation}. For recent discussion, see \citet{wagenmakers2012interpretation}, \citet{ding2016randomization} and \citet{keiding2016perils}. Therefore, if the treatment effect is more homogeneous under one scale, then it would be preferable to use this scale as a measure of the treatment effect for transportability across populations. For a binary outcome, according to \citet{poole2015risk}'s review, clinical reports and general guidelines often assert that the treatment effect is more heterogeneous under the risk difference scale than under the risk ratio and odd ratio scales. These claims in the literature are based on surveys of meta-analyses showing that rejections of the null hypotheses of homogeneity, or equivalently the absence of interaction, happen more often under the risk difference scale than under the risk ratio and odds ratio scales. However, the evidence for these claims are not satisfactory, because different rejection rates of the null hypotheses of homogeneity may simply due to different statistical powers of tests. This is true even when there is no effect in one of the two subpopulations so that there is effect heterogeneity on all scales and arguably to the same degree \citep{poole2015risk}.

Recognizing this gap in the literature, for binary treatment, covariate and outcome, we use results from differential geometry to quantitatively compare the relative homogeneity of the risk difference, risk ratio and odds ratio. Our comparison has three aspects. First, homogeneity, or no interaction, restricts the four outcome probabilities to be in the three dimensional space. For different effect measures, these spaces have different domains with different volumes. We find that the volume of the domain is the smallest for the risk difference, and is the largest for the odds ratio. This finding quantifies more exactly the corresponding fact that once three outcome probabilities are fixed, it is always possible to find a fourth that would lead to odds ratio homogeneity but that it is not always possible to do so for the risk ratio and risk difference scales. 
Second, although these three spaces of homogeneity have measure zero in the four dimensional space, they have positive volumes in the three dimensional space. Our calculation shows that the volume of the subspace is the smallest for the risk difference, and is the largest for the odds ratio. This theoretical result demonstrates that odds ratio homogeneity holds for relatively more values of the outcome probabilities than risk ratio homogeneity, which further holds for more values of the outcome probabilities than risk difference homogeneity. Third, we compare the volumes of the acceptance regions of Wald-type tests against the null hypotheses of homogeneity for the risk difference, risk ratio and odds ratio, finding that the acceptance region has the smallest volume for the risk difference, and has the largest volume for the odds ratio. This result thus partially also explains the reason why the null hypotheses of homogeneity were rejected more often for the risk difference than for the odds ratio, a phenomenon that arises repeatedly in surveys of meta-analyses \citep{katerndahl1999variability, engels2000heterogeneity, sterne2001funnel, deeks2002issues, deeks2003effect}.

The remainder of the paper proceeds as follows.  
Section \ref{sec::notation} introduces the notation and definitions for homogeneity and interaction, and states our evidence for the claim that the risk difference seems more heterogeneous than the risk ratio and odds ratio. For the risk difference, risk ratio and odds ratio,
Section \ref{sec::domains} compares the volumes of the domains of the spaces of homogeneity of the outcome probabilities;
Section \ref{sec::volumes} compares the volumes of the spaces of homogeneity of the outcome probabilities;
Section \ref{sec::stat-inference} compares the volumes of the acceptance regions of the Wald-type tests for the null hypotheses of homogeneity.  
We conclude in Section \ref{sec::discussion}.

\section{Notation, Definitions, and Conclusion}
\label{sec::notation}

Assume that we have binary treatments or covariates $G$ and $E$, and a binary outcome $Y.$ Define the outcome probability as $p_{ge} = \pr(Y=1\mid G=g,E=e)$ for $g,e=0,1.$ In some cases, $G$ and $E$ may be the gene and environment exposures, and we are interested in gene-environment interaction. In other cases, $G$ may be a binary covariate and $E$ a binary exposure or treatment, and we are interested in the differential treatment effects of $E$ on the outcome $Y$ given different levels of $G.$
Based on the four outcome probabilities $ \mathcal{P} =   (p_{11}, p_{10}, p_{01}, p_{00})$, we introduce concepts of interaction under different effect scales \citep{darroch1994synergism, vanderweele2015explanation}. For instance, we say that there is no interaction between $G$ and $E$ on the risk difference scale if
$$
p_{11} - p_{10}- p_{01} + p_{00} = 0;
$$
there is no interaction between $G$ and $E$ on the risk ratio scale if 
$$
\frac{p_{11} p_{00} }{  p_{10}p_{01} }  = 1;
$$
there is no interaction between $G$ and $E$ on the odds ratio scale if
$$
\frac{   p_{11}/(1-p_{11}) \times p_{00}/(1-p_{00})  }{  p_{10}/(1-p_{10}) \times p_{01}/(1-p_{01})    } = 1.
$$
We say that the treatment effect of $E$ on the outcome $Y$ is homogeneous across groups of $G$, if there is no interaction between $G$ and $E$. Clearly, the definition of homogeneity is scale dependent. It is possible that the treatment effect of $E$ on the outcome $Y$ is heterogeneous under all scales; it is also possible that the treatment effect of $E$ on the outcome $Y$ is homogeneous under one scale, but heterogeneous under another scale.

One important question is under which scale the treatment effect is more homogeneous. This question has profound impact in practice, because under a more homogeneous scale the treatment effect is more likely to be transportable to other subgroups or subpopulations. Below, we quantitatively compare the three effect scales in three ways, providing theoretical evidence for the claim that the risk difference is more heterogeneous than the risk ratio and odds ratio.

\section{Domains of the Homogeneity Spaces}
\label{sec::domains}

Assume $\mathcal{P}  \in (0, p)^4$, which lies in a four dimensional space. If we do not impose any restrictions on the ranges of the probabilities, then we can take $p=1$; if we consider only rarer outcomes, then we can take $p$ to be a smaller number, e.g., $p=0.1$.

In the absence of interaction on the risk difference scale, $p_{00} = -p_{11} + p_{10} +p_{01}$ is a deterministic function of $p_{11},p_{10}$ and $p_{01}.$  Because the probabilities are bounded between $0$ and $p$, this deterministic relationship restricts $p_{11},p_{10}$ and $p_{01}$ to be within the following region:
$$
\mathcal{F}_a(p) =  \{  (p_{11},p_{10}, p_{01}) \in (0, p)^3:    
\max(p_{10}+p_{01}-p,0) \leq p_{11} \leq  p_{10}+p_{01} \},
$$

In the absence of interaction on the risk ratio scale, $p_{00} = p_{10}p_{01} / p_{11}$ is a deterministic function of $p_{11},p_{10}$ and $p_{01}$, which have the following constraints:
$$
\mathcal{F}_m(p) =  \{  (p_{11},p_{10}, p_{01}) \in (0, p)^3:  
   p_{01} \leq  p_{11}p/p_{10}   \}.
$$

In the absence of interaction on the odds ratio scale, the deterministic relationship $p_{00}/(1-p_{00}) = p_{10} /(1-p_{10}) \times p_{01} / (1-p_{01}) / \{ p_{11} / (1-p_{11}) \}$ restricts $p_{11},p_{10}$ and $p_{01}$ to be within the following region:
$$
\mathcal{F}_o(p) = 
\left\{   (p_{11},p_{10}, p_{01}) \in (0, p)^3:  
    \frac{(1-p) p_{10}p_{01} }{   (1-p) p_{10}p_{01} +p(1-p_{10})(1-p_{01})  }  <  p_{11}    \right\}  .
$$

We can see that the homogeneity spaces have different domains for different effect scales.
When $p=1$, there is no restrictions on the outcome probabilities on the odds ratio scale, and $\mathcal{F}_o(1)$ is simply $(0,1)^3$; however, the restrictions on the outcome probabilities exist for any values of $p$ on the risk difference and risk ratio scales.

Define the volumes of the domains $\mathcal{F}_a(p) ,\mathcal{F}_m(p) $ and $\mathcal{F}_o(p) $ as $F_a (p), F_m(p)  $ and $F_o(p) $.
We first argue, based on the volumes of these domains, that the risk difference is more heterogeneous than the risk ratio and odds ratio.

\begin{theorem}
\label{thm::domain}
The three dimensional volume of the domain $\mathcal{F}_a(p)$ is 
$$
F_a (p) = 2/3\times p^3;
$$
the three dimensional volume of the domain $\mathcal{F}_m(p)$ is 
$$
F_m(p) = 3/4\times p^3;
$$
the three dimensional volume of the domain $\mathcal{F}_o(p) $ cannot be easily calculated explicitly, and the numerical values of $F_o(p)/p^3$ are tabulated at different values of $p$ below:
$$
\begin{pmatrix}
p&0.1&0.2&0.3&0.4&0.5&0.6&0.7&0.8&0.9&1\\
F_o(p)/p^3&0.75 &0.76 &0.76& 0.77& 0.77& 0.78& 0.79& 0.81 &0.85 &1.00
\end{pmatrix} . 
$$
\end{theorem}

We present the formula of $F_o(p)$ in the Appendix, and the numerical values of $F_o(p)$ are computed via both Monte Carlo and numerical integration. Note that $p^3$ is the reference volume of the region $(0,p)^3$. From the above theorem, at $p\leq 0.1$, the volumes of the domains are both $3/4$ under the risk ratio and odds ratio scales, because $p\leq 0.1$ corresponds to the rare outcome cases, and the risk ratio and odds ratio are very close. At $p=1$, the volumes of the domains are $2/3, 3/4$ and $1$ under the risk difference, risk ratio and odds ratio, respectively. 

\section{Volumes of the Homogeneous Spaces}
\label{sec::volumes}

In the absence of interaction on the risk difference scale, $\mathcal{P}$ lies in the following three dimensional space:
\begin{eqnarray*}
S_a(p) = \left\{   \mathcal{P}\in    (0, p)^4 : 
p_{11} + p_{00}  = p_{10} + p_{01}   \right\};
\end{eqnarray*}
in the absence of interaction on the risk ratio scale, $\mathcal{P}$ lies in the following three dimensional space:
\begin{eqnarray*} 
S_m(p) =  \left\{  \mathcal{P} \in    (0, p)^4 : 
p_{11} p_{00}  = p_{10}  p_{01}   \right\};
\end{eqnarray*}
in the absence of interaction on the odds ratio scale, $\mathcal{P}$ lies in the following three dimensional space:
\begin{eqnarray*}
S_o(p) =  \left\{   \mathcal{P} \in    (0, p)^4 : 
\frac{p_{11}}{1-p_{11}} \times \frac{p_{00}}{1-p_{00}}  = \frac{p_{10}}{1-p_{10}} \times  \frac{p_{01}}{1-p_{01}}   \right\}  .
\end{eqnarray*}

In the four dimensional space $(0,p)^4$, the three dimensional sub-spaces $S_a(p) , S_m(p) $ and $S_o(p) $ all have measure zero. However, it is reasonable to compare the three dimensional volumes, $V_a(p), V_m(p)$ and $V_o(p)$, of the three sub-spaces of the four dimensional space. See the Appendix for more technical discussion about the volume of a low dimensional space in a high dimensional space.

\begin{theorem}
\label{thm::three-dim-volume}
The three dimensional volume of the space $S_a(p)$ is
\begin{eqnarray*}
V_a(p)  = 4/3 \times p^3= 1.33\times p^3;
\end{eqnarray*}
the three dimensional volume of the space $S_m(p)$ is
\begin{eqnarray*}
V_m(p)  = \frac{1}{3} \left\{ \sqrt{2}+\log(1+\sqrt{2})  \right\}^2 \times p^3 = 1.76\times p^3;
\end{eqnarray*}
the three dimensional volume of the space $S_a(p)$ has a complicated form presented in the Appendix,  and $V_o(p)/p^3$ is tabulated at different values of $p$ below:
$$
\begin{pmatrix}
 p& 0.1 & 0.2 & 0.3 & 0.4 & 0.5 & 0.6 & 0.7 & 0.8 & 0.9 & 1 \\ 
  V_o(p)/p^3 & 1.76 & 1.76 & 1.76 & 1.77 & 1.78 & 1.79 & 1.81 & 1.85 & 1.93 & 2.47 \\ 
\end{pmatrix}.
$$
\end{theorem}

Again, for $p\leq 0.1$, the volumes of the homogeneity spaces for the risk ratio and odds ratio are the same; for any value of $p$, the volumes of the homogeneity spaces for the risk difference, risk ratio, and odds ratio are in an increasing order; at $p=1$, the three dimensional volumes of the homogeneity spaces are $1.33, 1.76$ and $2.47$ for the risk difference, risk ratio and odds ratio, respectively.

The differences of the three dimensional volumes are due to two reasons. First, the domains of the three spaces $\mathcal{F}_a(p) ,\mathcal{F}_m(p) $ and $\mathcal{F}_o(p) $ have different volumes as shown in Theorem \ref{thm::domain}. Second, the spaces $S_a(p) , S_m(p) $ and $S_o(p) $ have different ``curvatures'', which result in the differences of the volumes after taking account of the differences in the domains. We summarize this result in the following Corollary.

\begin{corollary} 
\label{coro::ratio}
The ratio $V_a(p) / F_a(p)$ is $2$; the ratio $V_m(p) / F_m(p)$ is
$$
\frac{4}{9}  \left\{    \sqrt{2} + \log(1+\sqrt{2})  \right\}^2 = 2.34;
$$
the ratio of $V_o(p) / F_o(p)$ is tabulated below:
$$
\begin{pmatrix}
 p& 0.1 & 0.2 & 0.3 & 0.4 & 0.5 & 0.6 & 0.7 & 0.8 & 0.9 & 1 \\ 
V_o(p) / F_o(p) &2.33 &2.32 &2.32 &2.31 &2.30 &2.29 &2.28& 2.27& 2.27& 2.47 
\end{pmatrix}.
$$
\end{corollary}

By comparing the ratio between the volumes of the three dimensional spaces and their domains, the risk difference is still the most heterogeneous.

%

\section{Heterogeneity Comparison Based on Statistical Inference}
\label{sec::stat-inference}

The above theoretical results discuss the properties of the effect scales ignoring issues of statistical inference.
In practice, there is another level of heterogeneity due to sampling varibility. With finite samples, we will never have exact no interaction under any effect scale. For $g,e=0,1$, let $\widehat{p}_{ge}$ be the sample analogue of $p_{ge}$, which has estimated asymptotic variance $\widehat{p}_{ge}(1-\widehat{p}_{ge})/n_{ge}$ with $n_{ge}$ being the sample size of group $(g,e)$. Taking this uncertainty into account, we can compare the volumes of the regions of $\widehat{\mathcal{P}} =  (\widehat{p}_{11}, \widehat{p}_{10}, \widehat{p}_{01}, \widehat{p}_{00}) \in [0,1]^4$, within which we fail to reject the following null hypotheses of homogeneity:
\begin{eqnarray*}
&H_a:& p_{11} + p_{00} = p_{10} + p_{01} ,\\
&H_m:& p_{11} p_{00}  = p_{10}  p_{01} ,\\
&H_o:& \frac{p_{11}}{1-p_{11}} \times \frac{p_{00}}{1-p_{00}}  = \frac{p_{10}}{1-p_{10}} \times  \frac{p_{01}}{1-p_{01}}. 
\end{eqnarray*}
We consider the Wald-type tests of significance level $\alpha$, with $C_\alpha = \{  \Phi^{-1}(\alpha/2) \}^2$ and $\logit(x)=\log\{x/(1-x) \}$. These three acceptance regions are
\begin{eqnarray*}
R_a&=& \left\{ \widehat{\mathcal{P}}  \in [0,1]^4  :  
\frac{   (\widehat{p}_{11}+\widehat{p}_{00}-\widehat{p}_{10}-\widehat{p}_{01} )^2 }
{\sum_{g,e=0,1}  \widehat{p}_{ge}(1-\widehat{p}_{ge})/n_{ge}}  
\leq C_\alpha     \right\} ,\\
R_m &=& \left\{  \widehat{\mathcal{P}}  \in [0,1]^4  :
 \frac{  (\log \widehat{p}_{11}+ \log \widehat{p}_{00}- \log \widehat{p}_{10}- \log \widehat{p}_{01} )^2}
 {\sum_{g,e=0,1}  (1-\widehat{p}_{ge})/(   n_{ge} \widehat{p}_{ge} ) }  
  \leq C_\alpha    \right\}  ,\\
R_o &=& \left\{   \widehat{\mathcal{P}}  \in [0,1]^4 :  
\frac{ ( \logit \widehat{p}_{11}+\logit \widehat{p}_{00}-\logit \widehat{p}_{10}-\logit \widehat{p}_{01} )^2 }
{\sum_{g,e=0,1}  1/\{  n_{ge} \widehat{p}_{ge}(1-\widehat{p}_{ge}) \} }
\leq C_\alpha    \right\} . 
\end{eqnarray*}

Although explicit formulas of these four dimensional volumes are not straightforward to obtain, it is easy to compute their numerical values by using Monte Carlo. Assume that $n_{ge}$'s have the same size $n$. Table \ref{tb::numerical-stat} shows the volumes of the regions $R_a, R_m$ and $R_o$ for different sample sizes, from which we can see that the acceptance region of the null hypothesis of homogeneity has the smallest volume for the risk difference, and has the largest volume for the odds ratio. This numerical result also partially explains why the rejection rates of the null hypotheses of homogeneity are higher for the risk difference than for the risk ratio and odds ratio, as found in the surveys of meta-analyses \citep{katerndahl1999variability, engels2000heterogeneity, sterne2001funnel, deeks2002issues, deeks2003effect, poole2015risk}.

\begin{table}  
\centering 
\caption{Volumes of acceptance regions $R_a, R_m$ and $R_o$ with different sample sizes}\label{tb::numerical-stat} 
\begin{tabular}{cccccc}
\hline 
 &  $n=100$ & $n=500$ & $n=2000$ & $n=5000$ &  $n=10000$ \\
 \hline 
$\Vol(R_a)$&  $0.214$ & $0.097$ & $0.049$ & $0.031$ & $0.022$ \\
$\Vol(R_m)$&  $0.246$& $0.107$ & $0.053$ & $0.034$ & $0.024$ \\
$\Vol(R_o)$& $0.253$ & $0.111$& $0.055$ & $0.035$ & $0.025$ \\
$\Vol(R_m)/\Vol(R_a)$&$1.148$& $1.106$ & $1.096$ & $1.093$ & $1.092$\\
$\Vol(R_o)/\Vol(R_a)$&$1.182$ & $1.151$ & $1.139$& $1.135$ & $1.134$\\
\hline 
\end{tabular}
\end{table}

\section{Discussion}
\label{sec::discussion}

It is often believed that the risk difference is more heterogeneous than the risk ratio and odds ratio. Previously, the evidence for this belief is based on the rejection rates of the tests against the null hypotheses of homogeneity for difference effect scales, which is not adequate either empirically or theoretically. Through theoretical calculations, we provide additional evidence for the claim that the risk difference is more heterogeneous than the risk ratio and odds ratio, by showing that the homogeneity space for the risk difference has the smallest volume of domain, the smallest three dimensional volume, and the smallest volume of the acceptance region based on the Wald-type test. 
However, to argue that in reality the risk difference is more heterogeneous using the theoretical calculations here we would effectively have to assume that the outcome probabilities are uniformly distributed, an assumption that may not correspond to the empirical distributions of the outcome probabilities encountered in practice. For practical problems, it may be more useful to gather some prior information about the distributions of the outcome probabilities, and compute the volumes of the domain or spaces weighted by these prior distributions. However, to do this, one would also have to specify the domains of the empirical settings under consideration (e.g., disciplines, exposures, outcomes, etc). A uniform distribution of outcome probability seems the only one natural choice a priori. 
However, empirical evidence, perhaps collected and compared across disciplines, might given further evidence as to which effect measure is the most heterogeneous.

\section*{Acknowledgment}

Peng Ding thanks Mr. Yi Xie in the Harvard Mathematics Department for helpful discussion on differential geometry. Tyler J. VanderWeele is funded by the U.S. National Institutes of Health grant ES017876.

\bibliographystyle{apalike}
\bibliography{cornfield}

\clearpage 
\appendix
\setcounter{page}{1}
\begin{center}
\title{\Large \bf Appendix}
\bigskip 

\date{}
\maketitle
\end{center}

\setcounter{equation}{0}
\setcounter{section}{0}
\setcounter{figure}{0}
\setcounter{example}{0}
\setcounter{proposition}{0}

\renewcommand {\theproposition} {A.\arabic{proposition}}
\renewcommand {\theexample} {A.\arabic{example}}
\renewcommand {\thefigure} {A.\arabic{figure}}
\renewcommand {\thetable} {A.\arabic{table}}
\renewcommand {\theequation} {A.\arabic{equation}}
\renewcommand {\thelemma} {A.\arabic{lemma}}
\renewcommand {\thesection} {A.\arabic{section}}

\section{A Lemma}

The following lemma is a known result in the literature, but we give an elementary proof.

\begin{lemma}\label{lemma::volume}
Assume that $(\alpha_1,\alpha_2, \ldots, \alpha_m)$ are $m$ linearly independent vectors in $R^n (m\leq n)$. Let $A$ be an $ n\times m$ matrix defined as $ A = [\alpha_1, \ldots, \alpha_m]$. The $m$-dimensional volume of the parallelepiped formed by $(\alpha_1, \ldots, \alpha_m)$ is $\sqrt{  \det(A^\top A) } $.
\end{lemma}

\begin{proof}[Proof of Lemma \ref{lemma::volume}.]
Let $C$ be the parallelepiped formed by $(\alpha_1, \ldots, \alpha_m)$. We use the following intuitive definition of the volume of a parallelepiped in a $m$-dim subspace of $n$-dim space:
$$
\text{vol}(C) = ||\alpha_1||\times ||  P_1^{\bot} \alpha_2 || \times \cdots \times  || P_{1:(m-1)}^\bot  \alpha_m ||,
$$
where $||\cdot||$ represents the $L-2$ norm, and $P_{1:(j-1)}^\bot \alpha_j$ is the projection of $\alpha_j$ onto the space orthogonal to the linear space spanned by $(\alpha_1, \cdots, \alpha_{j-1}).$
This definition is closely related to the Gram--Schmidt orthogonalization, or equivalently the QR decomposition of a matrix. We have
$$
A_{n\times m} = \Gamma_{n\times m}   U_{m\times m},
$$
where $\Gamma$ has orthonormal column vectors, and $U$ is an upper-triangular matrix. In the construction of the QR decomposition, the diagonal elements of $U$ are the lengths of $P_{1:(j-1)}^\bot \alpha_j$ for $j=1,\ldots,p$. Therefore, 
$$
\text{vol}(C) = \prod_{i=1}^p |  l_{ii}  |  = \det(L) = \sqrt{  \det(L^\top L)} = \sqrt{  \det(  L^\top \Gamma^\top  \Gamma L  )}
= \sqrt{ \det( A^\top A) }. 
$$
\end{proof}

\section{Computing $F_a(p)$ and $V_a(p)$}\label{sec::a-a}

For ease of notation, we use $(x,y,z,w) = (p_{11}, p_{10}, p_{01}, p_{00} ) $ and first assume $p=1$. We simply denote $S_a(1)$ and $V_a(1)$ by $S_a$ and $V_a$.  All vectors are column vectors. The set $S_a$ represents a mapping from a three dimensional set to the four dimensional space, written as 
$$  
\bm{f}_a(x,y,z) = (x,y,z, w=-x+y+z)^\top.
$$

First, we need to find the domain of the mapping $\bm{f}_a.$ Because $x,y,z, w = -x+y+z \in (0,1)$, the domain of $\bm{f}_a$ is
$$
\mathcal{F}_a = \{  (x,y,z):    \max(y+z-1,0) \leq x \leq \min(y+z,1) , 0\leq y\leq 1, 0\leq z\leq 1\}.
$$

Second, we need to find the infinitesimal volume spanned by the following three vectors corresponding to the directions of $(\d x,\d y,\d z)$:
\begin{eqnarray}\label{eq::basis-additive}
{ \partial \bm{f}_a \over  \partial x } \d x  = 
\begin{pmatrix}
1\\
0\\
0\\
-1
\end{pmatrix}\d x, \quad 
{ \partial \bm{f}_a \over  \partial y }  \d y  = 
\begin{pmatrix}
0\\
1\\
0\\
1
\end{pmatrix}\d y,\quad 
{ \partial \bm{f}_a \over  \partial z } \d z  =
\begin{pmatrix}
0\\
0\\
1\\
1
\end{pmatrix}\d z.
\end{eqnarray}
Let 
$$
A_a(x,y,z) = \left[{ \partial \bm{f}_a \over  \partial x }   , { \partial \bm{f}_a \over  \partial y } ,  { \partial \bm{f}_a \over  \partial z } \right] =
\begin{pmatrix}
1&0&0\\
0&1&0\\
0&0&1\\
-1&1&1
\end{pmatrix}.
$$
Lemma \ref{lemma::volume} implies that the infinitesimal volume spanned by the vectors in (\ref{eq::basis-additive}) is
\begin{eqnarray*}
v_a(x,y,z) \d x\d y\d z &=&\sqrt{ \det\{ A_a^\top(x,y,z) A_a(x,y,z) \}   } \d x \d y \d z \\
&=& \sqrt{    \det \left\{        
\begin{pmatrix}
1&0&0&-1\\
0&1&0&1\\
0&0&1&1
\end{pmatrix}
\begin{pmatrix}
1&0&0\\
0&1&0\\
0&0&1\\
-1&1&1
\end{pmatrix} \right\} 
     }   \d x \d y \d z \\
&=& \sqrt{   \det       
\begin{pmatrix}
2&-1&-1\\
-1&2&1\\
-1&1&2
\end{pmatrix}
 }   \d x \d y \d z \\
&=& 2      \d x \d y \d z .
\end{eqnarray*}

Therefore, the three dimensional volume of $S_a$ is 
\begin{eqnarray*}
V_a &=&  \iiint\limits_{\mathcal{F}_a}   v_a(x,y,z) \d x\d y\d z  = 2  \iiint\limits_{\mathcal{F}_a}  \d x\d y\d z
= 2\int_0^1 \int_0^1  \left\{     
\int_{  \max(y+z-1, 0)  }^{\min(y+z,1)} \d x
\right\} \d y \d z \\
&=& 2\int_0^1 \int_0^1  \left\{    
\min(y+z,1)  -  \max(y+z-1, 0)
\right\} \d y \d z \\
&=&2  \iint\limits_{  0<y,z<1, y+z<1 }  (y+z) \d y \d z
+ 2  \iint\limits_{  0<y,z<1, y+z>1 }  (2-y-z) \d y \d z\\
&=&2\int_0^1   \left\{  \int_0^{1-y}    (y+z)\d z   \right\}  \d y
+ 2\int_0^1 \left\{   \int_{1-y}^1   (2-y-z) \d z  \right\} \d y\\
&=& 2\int_0^1   \left\{  y(1-y) + \frac{1}{2}(1-y)^2  \right\}  \d y
+ 2\int_0^1 \left[   2y - y^2 - \frac{1}{2}\{  1- (1-y)^2   \}^2  \right]  \d y\\
&=& 4/3,
\end{eqnarray*}
where the last line follows from integrals of simple polynomials.

If the probabilities are bounded by general $p\in (0,1]$, then the domain becomes
$$
\mathcal{F}_a(p) =  \{  (x,y,z):    \max(y+z-p,0) \leq x \leq \min(y+z,p) , 0\leq y\leq p, 0\leq z\leq p\},
$$
and the volume becomes
$$
V_a(p) = 2 \iiint\limits_{\mathcal{F}_a (p)}    \d x\d y\d z .
$$
Applying the transformations $x'=x/p, y'=y/p$ and $z'=z/p$, the above integral becomes
$
V_a(p) = p^3 V_a(1).
$

From the above calculation, we know that $F_a(p) = V_a(p)/2 =  2/3\times p^3$.

\section{Computing $F_m(p)$ and $V_m(p)$}\label{sec::a-m}

Again we use $(x,y,z,w) = (p_{11}, p_{10}, p_{01}, p_{00} ) $, and first assume $p=1$. We simply denote $S_m(1)$ and $V_m(1)$ by $S_m$ and $V_m$. The set $S_m$ represents a mapping from a three dimensional set to the four dimensional space, written as 
$$  
\bm{f}_m(x,y,z) = (x,y,z, w= yz/x)^\top.
$$

First, we need to find the domain of the mapping $\bm{f}_m.$ Because $x,y,z, w =yz/x \in (0,1)$, the domain of $\bm{f}_m$ is
$$
\mathcal{F}_m = \{  (x,y,z):    yz < x < 1 , 0\leq y\leq 1, 0\leq z\leq 1\}.
$$
For convenience in later calculation, we express $\mathcal{F}_m$ in the following form:
$$
\mathcal{F}_m = \{  (x,y,z):    0\leq x\leq 1, 0\leq y\leq 1, 0< z < \min(x/y, 1)\}.
$$
It is straightforward to show that the above two forms of $\mathcal{F}_m $ are equivalent.

Second, we need to find the infinitesimal volume spanned by the following three vectors corresponding to the directions of $(\d x,\d y,\d z)$:
\begin{eqnarray}\label{eq::basis-m}
{ \partial \bm{f}_m \over  \partial x } \d x  = 
\begin{pmatrix}
1\\
0\\
0\\
-yz/x^2
\end{pmatrix}\d x, \quad 
{ \partial \bm{f}_m \over  \partial y }  \d y  = 
\begin{pmatrix}
0\\
1\\
0\\
z/x
\end{pmatrix}\d y,\quad 
{ \partial \bm{f}_m \over  \partial z } \d z  =
\begin{pmatrix}
0\\
0\\
1\\
y/x
\end{pmatrix}\d z.
\end{eqnarray}
Let 
$$
A_m(x,y,z) = \left[{ \partial \bm{f}_m \over  \partial x }   , { \partial \bm{f}_m \over  \partial y } ,  { \partial \bm{f}_m \over  \partial z } \right] =
\begin{pmatrix}
1&0&0\\
0&1&0\\
0&0&1\\
-yz/x^2& z/x& y/x
\end{pmatrix}.
$$
Lemma \ref{lemma::volume} implies that the infinitesimal volume spanned by the vectors in (\ref{eq::basis-m}) is
\begin{eqnarray*}
v_m(x,y,z) \d x\d y\d z &=&\sqrt{ \det\{ A_m^\top(x,y,z) A_m(x,y,z) \}   } \d x \d y \d z \\
&=& \sqrt{    \det \left\{        
\begin{pmatrix}
1&0&0&-yz/x^2\\
0&1&0&z/x\\
0&0&1&y/x
\end{pmatrix}
\begin{pmatrix}
1&0&0\\
0&1&0\\
0&0&1\\
-yz/x^2& z/x& y/x
\end{pmatrix} \right\} 
     }   \d x \d y \d z \\
&=& \sqrt{   \det       
\begin{pmatrix}
1+y^2z^2/x^4 & -yz^2/x^3 & - y^2z/x^3\\
-yz^2/x^3 & 1+z^2/x^2 & yz/x^2\\
-y^2z/x^3&yz/x^2&1+y^2/x^2
\end{pmatrix}
 }   \d x \d y \d z \\
&=& \sqrt{      (1+y^2/x^2) (1+z^2/x^2)       }      \d x \d y \d z .
\end{eqnarray*}

Therefore, the three dimensional volume of $S_m$ is 
\begin{eqnarray*}
V_m &=&    \iiint\limits_{\mathcal{F}_m} v_m(x,y,z) \d x\d y\d z 
=\iiint\limits_{\mathcal{F}_m}  \sqrt{      (1+y^2/x^2) (1+z^2/x^2)       }      \d x \d y \d z \\
&=& \int_0^1 \int_0^1   \sqrt{1+y^2/x^2}  \left\{
\int_0^{   \min(x/y, 1) }  \sqrt{1+z^2/x^2} \d z 
\right\}     \d x \d y  \\
&\stackrel{u=z/x}{=}&  \int_0^1 \int_0^1   \sqrt{x^2+y^2}  \left\{
\int_0^{   \min(1/y, 1/x) }  \sqrt{1+u^2} \d u 
\right\}     \d x \d y  \\
&=& 2 \iint\limits_{0<x<y<1} \sqrt{x^2+y^2}  \left\{
\int_0^{  1/y}  \sqrt{1+u^2} \d u 
\right\}     \d x \d y .
\end{eqnarray*}

Because
\begin{eqnarray*}
\int_0^{  1/y}  \sqrt{1+u^2} \d u  &=& \frac{1}{2} \left[   u \sqrt{1+u^2} 
+ \log(u + \sqrt{1+u^2})
 \right]_0^{1/y}  \\
&=& \frac{1}{2}   \left[   y^{-2} \sqrt{1+y^2} 
+ \log(y^{-1} + \sqrt{1+y^{-2}}) 
 \right] ,
\end{eqnarray*}
the volume $S_m$ becomes
\begin{eqnarray*}
V_m &=&  \iint\limits_{0<x<y<1} \sqrt{x^2+y^2}  \left[   y^{-2} \sqrt{1+y^2} 
+  \log(y^{-1} + \sqrt{1+y^{-2}}) 
 \right]   \d x \d y \\
 &=& \int_0^1   
 \left[   y^{-2} \sqrt{1+y^2} 
+ \log(y^{-1} + \sqrt{1+y^{-2}}) 
 \right]   
 \left(   \int_0^y   \sqrt{x^2+y^2}   \d x   \right) 
 \d y.
\end{eqnarray*}

Because
\begin{eqnarray*}
\int_0^y   \sqrt{x^2+y^2}   \d x &=& 
\frac{1}{2}\left[    
x\sqrt{x^2+y^2} + y^2 \log(x+\sqrt{x^2+y^2})
\right]_0^y \\
&=& \frac{1}{2}  \left[ 
y\sqrt{2y^2} + y^2\log(y+\sqrt{2y^2}) - y^2 \log y
\right]\\
&=& \frac{\sqrt{2}+\log(1+\sqrt{2})}{2} y^2 ,
\end{eqnarray*}
the volume of $S_m$ becomes
\begin{eqnarray*}
V_m&=&\frac{\sqrt{2}+\log(1+\sqrt{2})}{2}    \int_0^1   
 \left[     \sqrt{1+y^2} 
+ y^2 \log(y^{-1} + \sqrt{1+y^{-2} })  
  \right]
 \d y.
\end{eqnarray*}

Because of the following two integrals
\begin{eqnarray*}
\int_0^1 \sqrt{1+y^2} \d y 
&=&\frac{1}{2}  \left[   y\sqrt{1+y^2} + \log(y+\sqrt{1+y^2})   \right]_0^1  \\
&=& \frac{1}{2} [\sqrt{2} + \log(1+\sqrt{2})],\\
\int_0^1  y^2 \log( y^{-1} + \sqrt{1+y^{-2} } )  \d y 
&=& \frac{1}{6}  \left[  
y\sqrt{1+y^2} -\log(y+\sqrt{1+y^2}) + 2y^3\log( y^{-1} + \sqrt{1+y^{-2} } )
\right]_0^1\\
&=&\frac{1}{6} [\sqrt{2} + \log(1+\sqrt{2})],
\end{eqnarray*}
the volume $V_m$ is
$$
V_m =\left\{ \sqrt{2}+\log(1+\sqrt{2}) \right \}^2  \times \frac{1}{2}\times \left(  \frac{1}{2} + \frac{1}{6} \right)
=\frac{1}{3} \left\{ \sqrt{2}+\log(1+\sqrt{2})  \right\}^2.
$$

If the probabilities are bounded by general $p$, then the domain becomes
$$
\mathcal{F}_m(p) =  \{  (x,y,z):   0\leq x \leq p, 0\leq y\leq p, 0\leq z \leq \min( xp/y ,p)   \},
$$
and the volume becomes
$$
V_m(p) = 2 \iiint\limits_{\mathcal{F}_m (p)}   \sqrt{      (1+y^2/x^2) (1+z^2/x^2)       }    \d x\d y\d z .
$$
Applying the transformations $x'=x/p, y'=y/p$ and $z'=z/p$, the above integral becomes
$
V_m(p) = p^3 V_m(1).
$

The volume of the domain $\mathcal{F}_m(1)$ is
\begin{eqnarray*}
F_m(1) = \iiint\limits_{\mathcal{F}_m} \d x\d y\d z  
=  \int_0^1\int_0^1 \int_{yz}^1  \d x \d y\d z 
= \int_0^1\int_0^1 (1-yz) \d y \d z 
= \int_0^1 \left(1- \frac{z}{2} \right) \d z = 1-\frac{1}{4} = \frac{3}{4}.
\end{eqnarray*}
For general $p$, similar to the discussion in Section \ref{sec::a-a}, the volume of $\mathcal{F}_m(p)$ is $F_m(p) = 3/4\times p^3$.

\section{Computing $F_o(p)$ and $V_o(p)$}\label{sec::a-o}

Again we use $(x,y,z,w) = (p_{11}, p_{10}, p_{01}, p_{00} ) $. Define $h(u) = u/(1-u)$ and $g(c) = c/(1+c)$, with derivatives $ h'(u)  = 1/(1-u)^2$ and $g '(c) = 1/(1+c)^2.$ The set $S_o$ represents a mapping from a three dimensional set to a four dimensional space, written as 
$$  
\bm{f}_o(x,y,z) = \left(x,y,z, w=  g\left(  \frac{h(y)h(z)}{h(x)}  \right)   \right)^\top.
$$

First, we need find the domain of the mapping $\bm{f}_o.$ Because $h(x),h(y),h(z), h(w) =   h(y)h(z) / h(x)    \in (0, h(p))$, the domain of $\bm{f}_o$ is 
$$
\mathcal{F}_o(p) = 
\left\{  (x,y,z) :     \frac{(1-p)yz}{   (1-p)yz  +p(1-y)(1-z)  }  <  x < p, 0< y < p, 0< z < p    \right\}  .
$$

For notational simplicity, define
\begin{eqnarray}
k(x,y,z) &=& g ' \left(  \frac{h(y)h(z)}{h(x)}  \right) \frac{h(y)h(z)}{h(x)} 
=   \left( 1 +  \frac{h(y)h(z)}{h(x)}  \right)^{-2} \frac{h(y)h(z)}{h(x)}   =  \frac{h(x)h(y)h(z)}{     \{  h(x) + h(y)h(z)  \}^2     } , \nonumber\\
&&\label{eq::k}\\
l(u) &=&  \d ~ \log h(u) / \d u = \{ u(1-u) \}^{-1}  \label{eq::l}.
\end{eqnarray} 
In the following calculation, we write $k=k(x,y,z), h_x = h(x), h_y = h(y), h_z = h(z),  l_x=l(x), l_y = l(y)$ and $l_z = l(z).$

Second, we need to find the infinitesimal volume spanned by the following three vectors corresponding to the directions of $(\d x,\d y,\d z)$:
\begin{eqnarray}\label{eq::basis-odds}
{ \partial \bm{f}_o \over  \partial x } \d x  = 
\begin{pmatrix}
1\\
0\\
0\\
- kl_x
\end{pmatrix}\d x, \quad 
{ \partial \bm{f}_o \over  \partial y }  \d y  = 
\begin{pmatrix}
0\\
1\\
0\\
kl_y
\end{pmatrix}\d y,\quad 
{ \partial \bm{f}_o \over  \partial z } \d z  =
\begin{pmatrix}
0\\
0\\
1\\
kl_z
\end{pmatrix}\d z . 
\end{eqnarray}
 
Let 
$$
A_o(x,y,z) = \left[{ \partial \bm{f}_m \over  \partial x }   , { \partial \bm{f}_m \over  \partial y } ,  { \partial \bm{f}_m \over  \partial z } \right] =
\begin{pmatrix}
1&0&0\\
0&1&0\\
0&0&1\\
- k l_x& k l_y & kl_z 
\end{pmatrix}.
$$
Lemma \ref{lemma::volume} implies that the infinitesimal volume spanned by the vectors in (\ref{eq::basis-odds}) is
\begin{eqnarray*}
&&v_o(x,y,z) \d x\d y\d z \\
&=&\sqrt{ \det\{ A_o^\top(x,y,z) A_o(x,y,z) \}   } \d x \d y \d z \\
&=& \sqrt{    \det \left\{        
\begin{pmatrix}
1&0&0& - k l_x\\
0&1&0& k l_y\\
0&0&1& k l_z
\end{pmatrix}
\begin{pmatrix}
1&0&0\\
0&1&0\\
0&0&1\\
- k l_x& k l_y & kl_z 
\end{pmatrix} \right\} 
     }   \d x \d y \d z \\
&=& \sqrt{   \det       
\begin{pmatrix}
1+k^2l_x^2 & -k^2l_xl_y & -k^2 l_x l_z \\
-k^2 l_x l_y & 1+ k^2l_y^2 & k^2 l_y l_z \\
-k^2 l_xl_z & k^2l_yl_z & 1+k^2l_z^2
\end{pmatrix}
 }   \d x \d y \d z \\
&=& \sqrt{     1+k^2 (l_x^2 + l_y^2 + l_z^2)      }      \d x \d y \d z \\
&=&  \sqrt{     1+       \frac{ h_x^2 h_y^2 h_z^2  (l_x^2 + l_y^2 + l_z^2)  }{    (h_x + h_y h_z)^4   }   }      \d x \d y \d z .
\end{eqnarray*}

Using the definitions in (\ref{eq::k}) and (\ref{eq::l}), we express the three dimensional volume of $S_o$ as
\begin{eqnarray*}
V_o (p)&=&  \iiint\limits_{ \mathcal{F}_o(p) } v_o(x,y,z) \d x\d y\d z \\ 
&=&  \iiint\limits_{ \mathcal{F}_o(p) }  \sqrt{     1+       
\frac{  x^2(1-x)^2y^2(1-y)^2 +   y^2(1-y)^2z^2(1-z)^2 + x^2(1-x)^2z^2(1-z)^2        }{   \{   x(1-y)(1-z) + (1-x)yz  \}^4      }
   }      \d x \d y \d z .
\end{eqnarray*}

The volume of the domain $\mathcal{F}_o(1)$ is $F_o(1) = 1.$
For general $p$, the volume 
$
\mathcal{F}_o(p)
$
is
$$
F_o(p) =  \iiint\limits_{ \mathcal{F}_o(p) }  \d x\d y\d z,
$$
which is complex, but can be evaluated numerically.

\end{document}